\documentclass{amsart}
\usepackage{amssymb}
\usepackage{amsthm}
\usepackage{amscd}
\include{diagrams}

\newcommand{\Q}{{\mathbb Q}}

\newcommand{\Z}{{\mathbb Z}}
\newcommand{\F}{{\mathbb F}}

\newcommand{\Kts}{K^{\times\,2}}
\newcommand{\Lts}{L^{\times\,2}}
\newcommand{\Fts}{F^{\times\,2}}
\newcommand{\fa}{{\mathfrak a}}
\newcommand{\fb}{{\mathfrak b}}
\newcommand{\fc}{{\mathfrak c}}

\newcommand{\cH}{{\mathcal H}}

\newcommand{\fl}{{\mathfrak l}}
\newcommand{\fm}{{\mathfrak m}}
\newcommand{\fp}{{\mathfrak p}}

\newcommand{\cO}{{\mathcal O}}

\newcommand{\Hom}{\mbox{\rm Hom}\,}

\newcommand{\Gal}{\mbox{\rm Gal}\,}

\newcommand{\Cl}{\operatorname{Cl}}
\newcommand{\Clp}{\widetilde{\Cl}}

\newcommand{\disc}{\mbox{\rm disc}\,}

\newcommand{\im}{\operatorname{im}}
\newcommand{\Sel}{\operatorname{Sel}}
\newcommand{\eps}{\varepsilon}

\newcommand{\hra}{\hookrightarrow}
\newcommand{\lra}{\longrightarrow}
\newcommand{\Lra}{\Longrightarrow}
\newcommand{\too}{\longmapsto}
\newcommand{\la}{\langle}
\newcommand{\ra}{\rangle}

\def\rsp{\raisebox{0em}[2.6ex][1ex]{\rule{0em}{2ex}}}

\newcounter{lemmacount}[section]

\newtheorem{theorem}[lemmacount]{Theorem}
\newtheorem{prop}[lemmacount]{Proposition}
\newtheorem{lem}[lemmacount]{Lemma}

{}
{\bf}{\it}
{\bf}{}

\title{Selmer Groups and Quadratic Reciprocity}
\author{Franz Lemmermeyer}
\address{M\"orikeweg 1, 73489 Jagstzell}
\email{hb3@ix.urz.uni-heidelberg.de}

\begin{document}

\begin{abstract}
In this article we study the $2$-Selmer groups of number fields $F$
as well as some related groups, and present connections to the 
quadratic reciprocity law in $F$.
\end{abstract}

\maketitle

Let $F$ be a number field; elements in $F^\times$ that are ideal 
squares were called singular numbers in the classical literature.
They were studied in connection with explicit reciprocity laws,
the construction of class fields, or the solution of embedding
problems by mathematicians like Kummer, Hilbert, Furtw\"angler, 
Hecke, Takagi, Shafarevich and many others. Recently, the groups 
of singular numbers in $F$ were christened Selmer groups by H.~Cohen 
\cite{Coh} because of an analogy with the Selmer groups in the 
theory of elliptic curves (look at the exact sequence (\ref{ESel})
and recall that, under the analogy between number fields and
elliptic curves, units correspond to rational points, and
class groups to Tate-Shafarevich groups).  

In this article we will present the theory of $2$-Selmer groups 
in modern language, and give direct proofs based on class field
theory. Most of the results given here can be found in \S\S\ 61ff of 
Hecke's book \cite{Hec}; they had been obtained by Hilbert and
Furtw\"angler in the roundabout way typical for early class field 
theory, and were used for proving explicit reciprocity laws. Hecke, 
on the other hand, first proved (a large part of) the quadratic 
reciprocity law in number fields using his generalized Gauss sums
(see \cite{Berg} and \cite{LemRL2}), and then derived the existence 
of quadratic class fields (which essentially is just the calculation 
of the order of a certain Selmer group) from the reciprocity law. 

In Takagi's class field theory, Selmer groups were moved to the
back bench and only resurfaced in his proof of the reciprocity 
law. Once Artin had found his general reciprocity law, Selmer
groups were history, and it seems that there is no coherent 
account of their theory based on modern class field theory. 

Hecke's book \cite{Hec} is hailed as a classic, and it deserves
the praise. Its main claim to fame should actually have been 
Chapter VIII on the quadratic reciprocity law in number fields,
where he uses Gauss sums to prove the reciprocity law, then 
derives the existence of $2$-class fields, and finally proves
his famous theorem that the ideal class of the discriminant of
an extension is always a square. Unfortunately, this chapter is 
not exactly bedtime reading, so in addition to presenting Hecke's
results in a modern language I will also give exact references
to the corresponding theorems in Hecke's book \cite{Hec} in the 
hope of making this chapter more accessible.

The actual reason for writing this article, however, was that
the results on Selmer groups presented here will be needed for 
computing the separant class group of $F$, a new invariant that 
will be discussed thoroughly in \cite{LSep}, and for proving
a generalization of Scholz's reciprocity law to arbitrary number
fields in \cite{LSch}.

\section{Notation}
Let $F$ be a number field. The following notation will be used
throughout this article:
\begin{itemize}
\item $n$ is the degree $(F:\Q)$ of $F$. By $r$ and $s$ we denote the
      number of real and complex places of $F$; in particular,
      we have $n = r+2s$;
\item $F^\times_+$ is the subgroup of all totally positive elements
      in $F^\times = F \setminus \{0\}$;
\item for abelian groups $A$, $\dim A/A^2$ denotes the dimension of
      $A/A^2$ as a vector space over $\F_2$; note that 
      $(A:A^2) = 2^{\dim A/A^2}$;  
\item $E$ is the unit group of $F$, and $E^+$ its subgroup of
      totally positive units; observe that $\dim E/E^2 = r+s$;
\item $\Cl(F)$ and $\Cl^+(F)$ denote the class groups of $F$ in
      the usual and in the strict sense;
\item $\rho = \dim \Cl(F)/\Cl(F)^2$ and $\rho^+ = \dim \Cl^+(F)/\Cl^+(F)^2$
      denote the $2$-ranks of the class groups in the usual and in 
      the strict sense;
\item $\Cl_F\{4\}$ is the ray class group modulo $4$ in $F$, i.e.,
      the quotient of the group of ideals coprime to $(2)$ by the
      subgroup of principal ideals $(\alpha)$ with $\alpha \equiv 1 \bmod 4$.
      Similarly, $\Cl_F^+\{4\}$ is the ray class group modulo $4\infty$ in $F$.
\end{itemize}

\section{The Selmer Group}
\subsection{Definition of the Selmer Groups}
The $2$-Selmer group $\Sel(F)$ of a number field $F$ is defined as
$$ \Sel(F) = \{\alpha \in F^\times: (\alpha) = \fa^2\}/\Fts. $$

The elements $\omega \in F^\times$ with $\omega\Fts \in \Sel(F)$
are called singular in the classical literature (see e.g. Hecke 
\cite[\S\ 61, art. 4]{Hec}); Cohen \cite{Coh} calls them virtual
units. In fact we will see that if $F$ has odd class number, then  
$\Sel(F) \simeq E/E^2$.

The following lemma will allow us to define homomorphisms from
$\Sel(F)$ into groups of residue classes:

\begin{lem}\label{L4}
Let $\fm$ be an integral ideal in a number field $F$. Then every 
element in $\Sel(F)$ can be represented by an element coprime to $\fm$.
\end{lem}

\begin{proof}
Let $\alpha \Fts \in \Sel(F)$ and write $(\alpha) = \fa^2$. Now 
find an ideal $\fb$ coprime to $\fm$ in the ideal class $[\fa]$; 
then $\gamma \fa = \fb$, hence $\beta = \alpha \gamma^2$ satisfies 
$(\beta) = \fb^2$.
\end{proof}

Let us now introduce the following groups:
$$ M_4 = (\cO_F/4\cO_F)^\times, \quad 
   M^+ = F^\times/F^\times_+, \quad 
   M_4^+ = M_4 \oplus M^+. $$
Note that $F^\times/F^\times_+ \simeq (\Z/2\Z)^r$ via the signature map.
Moreover, the isomorpism $M_4/M_4^2 \simeq (\Z/2\Z)^n$ is induced by the 
map sending $\alpha \bmod (2)$ to $1+2\alpha \bmod (4)$. This implies
that $M_4^+/(M_4^+)^2 \simeq (\Z/2\Z)^{2r+2s}$.

By Lemma \ref{L4} we can define a map $\phi: \Sel(F) \lra M_4/M_4^2$ by
sending $\alpha F^{\times\,2}$, where $\alpha$ is chosen coprime
to $2$, to the class of $\alpha \bmod 4$. Now we define certain 
subgroups of $\Sel(F)$ via the exact sequences
$$ \begin{CD}
   1 @>>> \Sel^+(F) @>>> \Sel(F) @>>> M^+/(M^+)^2 \\
   1 @>>> \Sel_4(F) @>>> \Sel(F) @>>> M_4/M_4^2  \\
   1 @>>> \Sel_4^+(F) @>>> \Sel(F) @>>> M_4^+/(M_4^+)^2.
  \end{CD} $$

\subsection{Computation of the Selmer Ranks}
Now let $\fm$ be an arbitrary modulus, i.e. a formal product
of an integral ideal and some real infinite primes. There is
a natural projection $\Cl_F\{\fm\} \lra \Cl(F)$, and this 
induces an epimorphism 
$\nu: \Cl_F\{\fm\}/\Cl_F\{\fm\}^2 \lra \Cl(F)/\Cl(F)^2$.
The kernel of $\nu$ consists of classes $[(\alpha)]_\fm$
with $\alpha \in F^\times$ coprime to $\fm$. In fact, if
we denote the of coprime residue classes modulo $\fm$ by
$M_\fm$, and define a homomorphism 
$\mu: M_\fm/M_\fm^2 \lra \Cl_F\{\fm\}/\Cl_F\{\fm\}^2$
by sending the coset of the residue class $\alpha + \fm$ 
to the coset of the ideal class $[(\alpha)]_\fm \in \Cl_F\{\fm\}$,
then it is easily checked that $\mu$ is well defined, and that
we have $\ker \nu = \im \mu$.

The kernel of $\mu_\fm$ consists of all residue classes
$\alpha M_\fm^2$ for which $(\alpha)$ is equivalent to the
principal class modulo squares. If $(\alpha) = \gamma \fb^2$, then 
$\fb^2 = (\beta)$ is principal (hence $\beta \Fts \in \Sel(F)$)
and can be chosen in such a way that $\alpha \equiv \beta \bmod 4$; 
conversely, if $\alpha$ is congruent modulo $4$ to some element in 
the Selmer group, then $\alpha M_\fm^2 \in \ker \fm$. 
This shows that $\ker \fm$ equals the image of the map
$\Sel(F) \lra M_\fm/M_\fm^2$. 

By taking $\fm = \infty$, $4$, and $4\infty$ we thus get the 
exact sequences
$$ \begin{array}{cccccccccccc}
  1 & \lra & \Sel^+   & \lra & \Sel & \lra & M_+/2
                      & \lra & \Cl^+/2 & \lra & \Cl/2 & \lra 1, \\
  1 & \lra & \Sel_4   & \lra & \Sel & \lra & M_4/2
                      & \lra & \Cl\{4\}/2 & \lra & \Cl/2 & \lra 1, \\
  1 & \lra & \Sel_4^+ & \lra & \Sel & \lra & M_4^+/2
                      & \lra & \Cl^+\{4\}/2 & \lra & \Cl/2 & \lra 1,
  \end{array} $$
where $A/2$ denotes the factor group $A/A^2$.

Let us now determine the order of the Selmer groups.
The map sending $\alpha \in \Sel(F)$ to the ideal class 
$[\fa] \in \Cl(F)$ is a well defined homomorphism which 
induces an exact sequence
$$ \begin{CD}\label{ESel}
    1 @>>> E/E^2    @>>> \Sel(F)   @>>> \Cl(F)[2] @>>> 1.
   \end{CD} $$
Since $E/E^2 \simeq (\Z/2\Z)^{r+s}$, this implies 
$\Sel(F) \simeq (\Z/2\Z)^{\rho+r+s}$.

The group $\Sel_4(F)$ consists of all $\alpha \in F^\times$
modulo squares such that $F(\sqrt{\alpha}\,)/F$ is unramified
outside infinity; this shows that $\dim \Sel_4(F) = \rho^+$.
Similarly, the elements of $\Sel_4^+(F)$ correspond to quadratic
extensions of $F$ that are unramified everywhere, hence
$\dim \Sel_4^+(F) = \rho$. Finally, the first of the three 
exact sequences above shows that $\dim \Sel^+(F) = \rho^+ + s$.
We have proved

\begin{theorem}\label{TSel}
Let $F$ be a number field and let $\rho$ and $\rho^+$ denote the 
$2$-ranks of the class groups in the usual and in the strict sense.
The dimensions of the Selmer groups as vector spaces over $\F_2$ 
are given by the following table:
$$ \begin{array}{l|cccc}
\rsp \qquad A &    \Sel(F)   &  \Sel^+(F) & \Sel_4(F) & \Sel_4^+(F) \\ \hline
\rsp  \dim  A & \rho + r + s & \rho^+ + s & \rho^+    & \rho 
   \end{array} $$
Similarly, the dimensions of the associated ray class groups are
$$ \begin{array}{l|cccc}
\rsp \qquad A & \Cl(F) & \Cl^+(F) & \Cl_F\{4\}  & \Cl_F^+\{4\} \\ \hline
\rsp \dim A/A^2 &  \rho  & \rho^+   & \rho^+ + s  & \rho + r + s
   \end{array} $$
\end{theorem}

The numbers in these tables suggest a duality between certain 
Selmer and ray class groups. We will see below that this is 
indeed the case: as a matter of fact, this duality is a simple
consequence of the quadratic reciprocity law.

Let me also mention that $\dim \Sel_4^+ = \rho$ is the
existence theorem for quadratic Hilbert class fields,
since it predicts that the maximal elementary abelian 
unramified $2$-extension of $F$ is generated by the 
square roots of $\rho = \dim \Cl(F)/\Cl(F)^2$ elements of $F$.

\section{Associated Unit Groups}
In analogy to the subgroups $\Sel^*(F)$ of the Selmer group
we can define subgroups of $E^*/E^2$ of $E/E^2$ as follows:
\begin{align*}
  E^+    & = \{\eps \in E: \eps \gg 0\}, \\
  E_4\,  & = \{\eps \in E: \eps \equiv \xi^2 \bmod 4\}, \\
  E_4^+  & = \{\eps \in E: \eps \equiv \xi^2 \bmod 4, \ \eps \gg 0\}. 
\end{align*}

Applying the snake lemma to the diagram
$$ \begin{CD}
   1 @>>> E/E^2  @>>> \Sel(F) @>>> \Cl(F)[2] @>>> 1 \\
   @.   @VVV @VVV @VVV @. \\
   1 @>>> M_4/M_4^2 @>>> M_4/M_4^2 @>>> 1 
   \end{CD} $$
provides us with an exact sequence
$$ \begin{CD}
   1 @>>> E_4/E^2 @>>> \Sel_4(F) @>>> \Cl(F)[2]; \end{CD}$$
This (and a similar argument involving $M^+$ instead of $M_4$) implies 

\begin{prop}
If $F$ is a number field with odd class number, then
$$ E/E^2 \simeq \Sel(F), \quad
   E_4/E^2 \simeq \Sel_4(F) \quad \text{and} \quad 
   E^+/E^2 \simeq \Sel^+(F). $$
\end{prop}

Now consider the natural map $\pi: \Cl^+(F) \lra \Cl(F)$ sending
an ideal class $[\fa]_+$ to the ideal class $[\fa]$; this 
homomorphism is clearly surjective. This gives the exact sequence
$$ \begin{CD}
   1 @>>> \ker \pi  @>>> \Cl^+(F) @>>> \Cl(F) @>>> 1,  
   \end{CD} $$
where $\ker \pi$ is the group of all ideal classes $[\fa]_+$ 
in the strict sense such that $[\fa] = 1$, i.e., 
$\ker \pi = \{[(\alpha)]: \alpha \in F^\times\}$.

Next consider the map $\eta: M^+ =  F^\times/F^\times_+ \lra \ker \pi$
defined by sending $\alpha F^\times_+$ to $[(\alpha)]_+$; this 
map is well defined and surjective, and its kernel consists of 
classes $\alpha F^\times_+$ that are represented by units, that 
is, $\ker \eta = EF^\times_+/F^\times_+ \simeq E/E^+$. Thus we
have
$$ \begin{CD}
   1 @>>> E/E^+ @>>> M^+ @>>> \ker \pi @>>> 1
   \end{CD} $$
Glueing the last two exact sequences together we get the exact sequence
\begin{equation}\label{ECC}
   \begin{CD}
   1 @>>> E/E^+ @>>> M^+ @>>> \Cl^+(F) @>>> \Cl(F) @>>> 1.  
   \end{CD} \end{equation}
This shows

\begin{prop}
We have $h^+(F) = 2^{r-u}h(F)$, where $u = \dim E/E^+$. 
\end{prop}

Thus whereas Selmer groups measure the difference of the ranks of
$\Cl^+(F)$ and $\Cl(F)$, the unit group contains information about
their cardinalities. Trying to extract information on $\rho^+ - \rho$
from the sequence (\ref{ECC}) does not work: note that 
$\Hom(\Z/2\Z,A) \simeq A[2] = \{a \in A: 2a = 0\}$ 
for an additiviely written abelian group $A$. Since 
$\Hom(\Z/2\Z,\,\cdot\,)$ is a left exact functor, and since
$E/E^+$ and $M^+$ are elementary abelian $2$-groups,
(\ref{ECC}) provides us with the exact sequence
$$ \begin{CD}
   1 @>>> E/E^+ @>>> F^\times/F^\times_+ 
     @>>> \Cl^+(F)[2] @>{\pi}>> \Cl(F)[2], \end{CD} $$
where we denoted the restriction of $\pi$ to $\Cl^+(F)[2]$ also
by $\pi$, and where exactness at $\Cl^+(F)[2]$ is checked directly.
Now 
$$\im \pi = \Clp(F) := 
    \{[\fa] \in \Cl(F): \fa^2 = (\alpha), \alpha \gg 0\}, $$
hence we find 

\begin{prop}\label{PCC}
The sequence
$$ \begin{CD}
   1 @>>> E/E^+ @>>> F^\times/F^\times_+ 
     @>>> \Cl^+(F)[2] @>{\pi}>> \Clp(F)@>>> 1 \end{CD} $$
is exact; in particular, we have $\dim \Clp(F) = \rho^+ - r + u$.
\end{prop}

\section{Applications}\label{Sapp}
\subsection{Unit Signatures}
Lagarias observed in \cite{Lag1} that the residue class
modulo $4$ of an element $\alpha \in \cO_F)$ with
$\alpha \Fts \in \Sel(F)$ determines its signature
for quadratic fields $F = \Q(\sqrt{d}\,)$, where
$d = x^2 + 16y^2$. This observation was generalized
in \cite{Lag2,Lag3}; the main result of \cite{Lag3} 
is the equivalence of conditions (1) -- (4) of the
following theorem:

\begin{theorem}
Let $F$ be a number field and put $\rho_4 = \dim \Cl_F\{4\}/\Cl_F\{4\}^2$. 
Then the following assertions are equivalent:
\begin{enumerate}
\item $s=0$, and the image of $\alpha \Fts \in \Sel(F)$ in $M_4/M_4^2$
      determines its signature;
\item $s = 0$ and $\rho^+ = \rho$;
\item $s = 0$ and $\Sel_4(F) \subseteq \Sel^+(F)$;
\item $s = 0$ and the map $\Sel(F) \lra M^+$ is surjective;
\item the image of $\alpha \Fts \in \Sel(F)$ in $M^+$ determines 
      its residue class modulo $4$ up to squares; 
\item $\rho_4 = \rho$;
\item $\Sel^+(F) \subseteq \Sel_4(F)$;
\item the map $\Sel(F) \lra M_4/M_4^2$ is surjective.
\end{enumerate}
\end{theorem}

Actually all these assertions essentially establish the
following exact and commutative diagram (for number fields 
$F$ with $s = \rho^+ - \rho = 0$):
$$ \begin{CD}
   1 @>>> \Sel^+(F) @>>> \Sel(F) @>>> M^+ @>>> 1 \\
   @.  @VVV   @VVV   @VVV  @. \\
   1 @>>> \Sel_4(F) @>>> \Sel(F) @>>> M_4/M_4^2 @>>> 1,
  \end{CD} $$
here the two vertical maps between the Selmer groups
are the identity maps. Conversely, this diagram 
immediately implies each of the claims (1) -- (8)
above.
   
\begin{proof}
Consider the exact sequence
$$  1 \lra  \Sel^+  \lra  \Sel  \lra  M_+
      \lra  \Cl^+(F)/\Cl^+(F)^2  \lra  \Cl(F)/\Cl(F)^2  \lra 1. $$
Clearly $\rho^+ = \rho$ if and only if the map $M^+ \lra  \Cl^+/2$
is trivial, that is, if and only if $\Sel(F) \lra M_+$ is surjective.
This proves (2) $\iff$ (4). 

Similarly, the exact sequence
$$ 1 \lra \Sel_4 \lra \Sel \lra M_4/M_4^2 
             \lra \Cl_F\{4\}/\Cl_F\{4\}^2 \lra \Cl(F)/\Cl(F)^2 \lra 1 $$
shows that (6) is equivalent to (8).

Theorem \ref{TSel} immediately shows that (2) $\iff$ (6).

(2) $\Lra$ (3) \& (7): Since $\Sel_4^+(F) \subseteq \Sel^+(F)$
  and both groups have the same dimension $\rho^+ + s = \rho$, 
  we conclude that $\Sel_4^+(F) = \Sel^+(F)$. A similar
  argument shows that $\Sel_4^+(F) = \Sel_4(F)$.

(3) $\Lra$ (5): assume that $\alpha \equiv \beta \bmod 4$;
 then $\alpha/\beta \in \Sel_4(F) \subseteq \Sel^+(F)$, 
 hence $\alpha$ and $\beta$ have the same signature.

(5) $\Lra$ (7): assume that $\alpha \Fts \in \Sel^+(F)$. 
 Then $\alpha$ and $1$ have the same signature, hence
 they are congruent modulo $4$ up to squares, and this 
 shows that $\alpha \Fts \in \Sel_4(F)$.

(7) $\Lra$ (2): $\Sel^+(F) \subseteq \Sel_4(F)$ implies
  $\Sel^+(F) \subseteq \Sel_4(F) \cap \Sel^+(F) = \Sel_4^+(F)$,
  and now Theorem \ref{TSel} shows that $s=0$ and $\rho^+ = \rho$. 

It remains to show that (1) $\iff$ (3). We do this in two steps.

(1) $\Lra$ (3): Assume that $\alpha \Fts \in \Sel_4(F)$. By 
Lemma \ref{L4} we may assume that $\alpha$ is coprime to $2$,
and hence that $\alpha \equiv \xi^2 \bmod 4$. Since the
residue class determines the signature, $\alpha$ has the
same signature as $\xi^2$, i.e., $\alpha$ is totally positive.

(3) $\Lra$ (1): Assume that $\alpha\Fts, \beta\Fts \in \Sel(F)$,
and that $\alpha \equiv \beta \bmod 4$. Then 
$\alpha/\beta \in \Sel_4(F) \subseteq \Sel^+(F)$, hence $\alpha$
and $\beta$ have the same signature.
\end{proof}

\subsection{The Theorem of Armitage-Fr\"ohlich}
As a simple application of Hecke's results on Selmer groups
we present a proof of the theorem of Armitage and Fr\"ohlich
on the difference between the class groups in the usual and 
in the strict sense. 

We make use of a group theoretical lemma:

\begin{lem}\label{Lgt}
Assume that $A$ is a finite abelian group with subgroups
$B$, $C$ and $D = B \cap C$. Then the inclusions $C \hra A$ and 
$D \hra B$ induce a monomorphism $C/D \lra A/B$; in particular,
we have $(C:D) \mid (A:B)$. 
\end{lem}

The proof of this lemma is easy. Applying it to $A = \Sel(F)$, 
$B = \Sel_4(F)$, $C = \Sel^+(F)$ and 
$D = \Sel_4^+(F) = \Sel_4(F) \cap \Sel^+(F)$ we get 
$\rho^+-\rho \le \rho - \rho^+ +r$, which gives

\begin{theorem}[Theorem of Armitage-Fr\"ohlich]\label{TAF}
Let $F$ be a number field with $r$ real embeddings. Then the
difference of the $2$-ranks of the class groups in the strict
and in the usual sense is bounded by $\frac{r}2$; since this
difference is an integer, we even have
$$ \rho^+ - \rho \le \Big\lfloor \frac{r}2 \Big\rfloor. $$
\end{theorem}

This proof of the theorem of Armitage \& Fr\"ohlich \cite{AF} 
is essentially due to Oriat \cite{Ori}. A proof dual to Oriat's 
was given by Hayes \cite{Hay}, who argued using the Galois 
groups of the Kummer extensions corresponding to elements in 
$\Sel(F)$.

Applying the lemma to $A =  \Sel^+(F)$, $B = \Sel_4^+(F)$,
$C = E^+$ and $D = E_4^+$ and using the theorem of Armitage-Fr\"ohlich
we find

\begin{theorem}\label{TG}
Let $F$ be a number field with $r$ real embeddings. Then
$$ \dim E_4^+/E^2 \ge \Big\lceil \frac{r}2 \Big\rceil - \dim E/E^+. $$
\end{theorem}
 
According to Hayes \cite{Hay}, this generalizes results
of Greither (unpublished) as well as Haggenm\"uller \cite{Hag}.

Let us now give a simple application of these results. Consider
a cyclic extension $F/\Q$ of prime degree $p$, and assume that
$2$ is a primitive root modulo $p$. Since the cyclic group
$G = \Gal(F/\Q)$ acts on class groups and units groups, we find
(see e.g. \cite{LGA}) that the dimensions of the $2$-class groups
$\Cl_2(F)$ and $\Cl_2^+(F)$, as well as of $E^+/E^2$ and $E_4^+/E^2$
(note that $G$ acts fixed point free on $E^+/E^2$, but not on $E/E^2$) 
as $\F_2$-vector spaces are all divisible by $p-1$. Since 
$\rho^+ - \rho \le \frac{p-1}2$ by Armitage-Fr\"ohlich, we conclude
that $\rho^+ = \rho$. This shows

\begin{prop}\label{POs}
Let $F$ be a cyclic extension of prime degree $p$ over $\Q$, and 
assume that $2$ is a primitive root modulo $p$. Then $\rho^+ = \rho$,
and in particular $F$ has odd class number if and only if there
exist units of arbitrary signature.
\end{prop}

Here are some numerical examples. For primes $p \equiv 1 \bmod n$,
let $F_n(p)$ denote the subfield of degree $n$ of $\Q(\zeta_p)$.
Calculations with {\tt pari} \cite{gp} provide us with the following table:
$$ \begin{array}{cr|cc}
 \rsp n &   p  & \Cl_2(F)   & \Cl_2^+(F) \\ \hline 
 \rsp 3 &  163 &  (2,2)     & (2,2) \\
 \rsp   & 1009 &  (2,2)     & (4,4) \\
 \rsp   & 7687 &  (2,2,2,2) & (2,2,2,2) \\
 \rsp 5 &  941 &  (2,2,2,2) & (2,2,2,2) \\
 \rsp   & 3931 &  (4,4,4,4) & (4,4,4,4) \\
 \rsp 7 &   29 &   1        & (2,2,2) \\
 \rsp   &  491 &  (2,2,2)   & (2,2,2,2,2,2)
\end{array} $$

Now assume that $h^+ > h$, where $h$ and $h^+$ denote the class
numbers of $F$ in the usual and the strict sense. In this case, 
$\dim E^+/E^2 > 0$, hence $\dim E^+/E^2 = p-1$ and therefore 
$\dim E/E^+ = \dim E/E^2 - \dim E^+/E^2 = p - (p-1) = 1$; in 
particular, $-1$ generates $E/E^+$. Using Theorem \ref{TG}
we find that $\dim E_4^+/E^2 \ge \lceil p/2 \rceil - 1 = \frac{p+1}2$,
and now the Galois action implies that $\dim E_4^+/E^2 = p-1$. Thus 
$E^+ = E_4^+$, hence $F(\sqrt{E^+})/F$ is a subfield of degree
$2^{p-1}$ of the Hilbert $2$-class field of $F$.

\begin{prop}
Let $F$ be a cyclic extension of prime degree $p$ over $\Q$, and 
assume that $2$ is a primitive root modulo $p$. If $h^+ > h$,
then every totally positive unit is primary, and $F(\sqrt{E^+})/F$ 
is a subfield of degree $2^{p-1}$ of the Hilbert class field of $F$.
\end{prop}

\section{Hecke's Presentation}
We will now explain how Hecke's results in \cite[\S\ 61]{Hec} are 
related to those derived above. The numbers below refer to the 13 
articles in \S\ 61 of Hecke's book:
\begin{enumerate}
\item[1.] $\dim E/E^2 = m := r+s$; (Hecke uses $r_1$ and $r_2$ instead
              of $r$ and $s$);
\item[2.] $\dim F^\times/F^\times_+ = r$;
\item[3.] $\rho = \dim \Cl(F)/\Cl(F)^2$; 
              (Hecke uses $e$ instead of $\rho$);
\item[4.] $\dim \Sel(F) = \rho+r+s$;
\item[5.] $p := \dim \Sel^+(F)$; the image of $\Sel(F)$ in
              $M^+$ has dimension $m+e-p$;
\item[6.] $\rho ^+ = \dim \Cl^+(F)/\Cl^+(F)^2$;
              $p = \rho^+ - r + m = \rho^+ + s$;
\item[7.] $\dim M_4/M_4^2 = n$;
\item[8.] $\dim M_4^+/(M_4^+)^2 = n+r = 2r+2s$;
\item[9.] Hecke introduces the group $M_{4\fl}$ of residue classes 
              modulo $4\fl$ for prime ideals $\fl \mid 2$ and proves
              that $\dim M_{4\fl}/M_{4\fl}^2 = n+r+1$;
\item[10.] $q   := \dim \Sel_4(F)$; $q \le \rho+r+s$;
\item[11.] $q_0 := \dim \Sel_4^+(F)$; 
\item[12.] $\dim \Cl_F\{4\}/\Cl_F\{4\}^2 = 2^{q+s}$;
\item[13.] $\dim \Cl_F^+\{4\}/\Cl_F^+\{4\}^2 = 2^{q_0+r+s}.$
\end{enumerate}
In \S\ 62 Hecke then uses analytic methods (and the quadratic
reciprocity law) to prove that $q = \rho^+$ and $q_0 = \rho$.

\section{Class Fields}
In this section we will realize the Kummer extensions
$F(\sqrt{\Sel^*(F)}\,)/F$ as class fields. 

\subsection{Selmer Groups and Class Fields}
As a first step, we determine upper bounds for the conductor of 
these extensions. To this end, we recall the conductor-discriminant 
formula. For quadratic extensions $K/F$ with $K = F(\sqrt{\omega}\,)$, 
it states that the discriminant of $K/F$ coincides with its
conductor, which in turn is defined as the conductor of the
quadratic character $\chi_\omega = (\frac{\omega}{\cdot})$.

\begin{prop}\label{Pcond}
Consider the Kronecker character $\chi = (\frac{\omega}{\cdot})$, 
where $\omega \in \cO_F$. Then $\chi$ is defined modulo $\fm$ if 
and only if $\omega$ satisfies the conditions $(*)$, and the 
elements $\omega \Fts$ with $\omega$ satisfying $(*)$ form a 
group denoted by $(\dagger)$:
$$ \begin{array}{c|l|l}
    \fm      & \qquad (*) & \quad (\dagger) \\  \hline
\rsp  (4)\infty  & (\omega) = \fa^2 & \Sel(F) \\
\rsp  (4)        & (\omega) = \fa^2, \omega \gg 0 & \Sel^+(F) \\
\rsp   \infty    & (\omega) = \fa^2, \omega \in M_4^2 & \Sel_4(F) \\
\rsp   1         & (\omega) = \fa^2, \omega \gg 0, \omega \in M_4^2 
                                                    & \Sel_4^+(F)
\end{array} $$
\end{prop}

\begin{proof}
Every prime ideal with odd norm dividing $(\omega)$ to an odd power 
divides the relative discriminant of $K = F(\sqrt{\omega}\,)$ to an 
odd power (cf. \cite[Satz 119]{Hec}); if $\fl \mid 2$ is a prime 
ideal dividing $(\omega)$ to an odd power, then 
$\fl^{2e+1} \parallel \disc K/F$, where $\fl^e \parallel (2)$. 
Thus if the conductor of $(\omega/\cdot)$ divides $4\infty$, then 
$(\omega)$ must be an ideal square. 

The extension $F(\sqrt{\omega}\,)/F$ is unramified at infinity if
and only if $\omega \gg 0$. It is unramified at $2$ if and only if
$\omega$ is a square modulo $4$, i.e., if and only if $\omega \in M_4^2$.
\end{proof}

This shows that $F(\sqrt{\Sel(F)}\,)$ is contained in the
ray class field modulo $4\infty$. Now let $H(F)$, $H^+(F)$, 
$H_4(F)$, and $H_4^+(F)$ denote the maximal elementary abelian 
$2$-extension of $F$ with conductor dividing $1$, $\infty$, $4$, 
and $4\infty$, respectively. 

If we put
\begin{align*}
  P^+ & = \{(\alpha) \in P: \alpha \gg 0\}, \\
  P_4 & = \{(\alpha) \in P: 
            (\alpha,2) = (1), \ \alpha \equiv \xi^2 \bmod 4\}, \\
  P_4^+ & = P^+ \cap P_4,
\end{align*}
then $P$, $P^+$, $P_4$ and $P_4^+$ are the groups of principal 
ideals generated by elements $\alpha \equiv 1 \bmod \fm$ with 
$\fm = 1$, $\infty$, $(4)$, and $(4)\infty$, respectively. Now we claim

\begin{theorem}
The class fields $H^*(F)$ can be realized as Kummer extensions
$H^*(F) = F(\sqrt{\Sel^*(F)}\,)$ generated by elements of the
Selmer group $\Sel^*(F)$. The ideal groups $\cH^*(F)$ associated
to the extensions $H^*(F)/F$ are also given in the table below:
$$ \begin{array}{c|ccc}
\rsp    H^*(F)   & \Sel^*(F)   & \cH^*(F)  & \fm \\ \hline 
\rsp    H(F)     & \Sel_4^+(F) & I^2 P     & (4)\infty \\
\rsp    H^+(F)   & \Sel_4(F)   & I^2 P^+   & (4) \\
\rsp    H_4(F)   & \Sel^+(F)   & I^2 P_4   & \infty \\
\rsp    H_4^+(F) & \Sel(F)     & I^2 P_4^+ & 1  
\end{array} $$
\end{theorem}

The diagram in Fig. \ref{F1} helps explain the situation.

\begin{figure}[ht!]
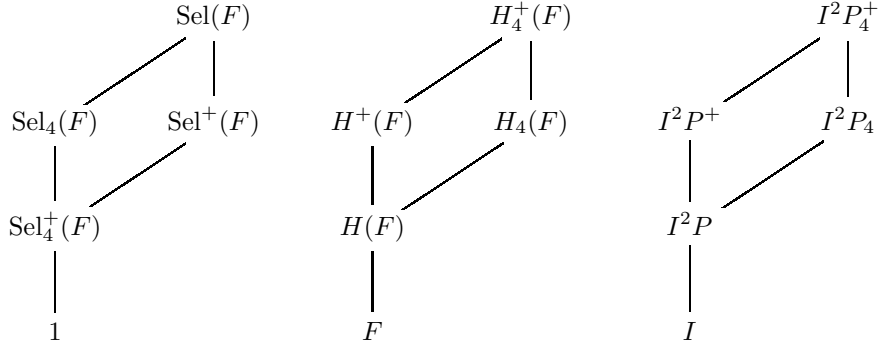

\begin{diagram}[height=.7cm] 
   &             &         & \Sel(F)  &  &         &         & H_4^+(F) &
   &         &         & I^2P_4^+  & \\ 
   &             & \ruLine & \dLine   &  &         & \ruLine & \dLine   &
   &         & \ruLine & \dLine   & \\
   & \Sel_4(F)   &         & \Sel^+(F)&  &  H^+(F) &         &  H_4(F)  & 
   &  I^2P^+ &         &  I^2P_4  & \\  
   &  \dLine     & \ruLine &          &  &  \dLine & \ruLine &          & 
   &  \dLine & \ruLine &          & \\
   & \Sel_4^+(F) &         &          &  &  H(F)   &         &          &
   &  I^2P   &         &          & \\
   &    \dLine   &         &          &  &  \dLine &         &          &
   &  \dLine &         &          & \\
   &      1      &         &          &  &    F    &         &          & 
   &    I    &         &          & 
\end{diagram}
\caption{Class Fields}\label{F1}
\end{figure}

\begin{proof}
The entries in the second (and the last) column follow immediately 
from Prop. \ref{Pcond}. It remains to compute the ideal groups
associated to the extensions $F(\sqrt{\Sel^*(F)}\,)/F$.

The ideal group associated to the Hilbert class field $H^1(F)$
is the group $P = P_F$ of principal ideals in $F$. Let $\cH$ denote the
ideal group associated to the maximal elementary abelian $2$-extension
$H(F)/F$; then $P \subseteq \cH \subseteq I$, where
$I = I_F$ is the group of fractional ideals in $F$, and 
$\cH$ is the minimal such group for which $I/\cH$ is an 
elementary abelian $2$-group. Clearly $\cH$ contains $I^2P$;
on the other hand, $I^2P/P \simeq \Cl(F)^2$, hence
$I/I^2P \simeq \Cl(F)/\Cl(F)^2$ has the right index, and 
we conclude that $\cH = I^2P$ (see Lagarias \cite[p. 3]{Lag2}).

Analogous arguments show that the ideal groups associated to the
maximal elementary abelian $2$-extensions with conductor dividing
$\infty$, $4$ and $4\infty$ are $I^2P^+$, $I^2P_4$ and $I^2P_4^+$,
respectively.
\end{proof}

If $\{\omega_1, \ldots, \omega_\rho\}$ is a basis of $\Sel_4^+(F)$
as an $\F_2$-vector space, and if $\sigma_j$ denotes the Legendre 
symbol $(\omega_j/\,\cdot\,)$, then the Artin symbol of $H(F)/F$
can be written as
$(H(F)/F,\,\cdot\,) = (\sigma_1, \ldots, \sigma_\rho)$.
Since the Artin symbol is defined for all unramified prime ideals
we have to explain what $(\omega/\fa)$ should mean if $\fa$ and
$\omega$ are not coprime. Using Lemma \ref{L4}, for evaluating 
$(\omega/\fa)$ we choose $\omega \Fts = \omega' \Fts$ with 
$(\omega') + \fa = (1)$ and put $(\omega/\fa) := (\omega'/\fa)$.

As is well known, the Artin symbol defines a isomorphism
between the ray class group associated to $H(F)/F$ and the
Galois group $\Gal(H(F)/F) \simeq (\Z/2\Z)^\rho$.
This shows that the kernel of the Artin symbol 
$(H(F)/F,\,\cdot\,): I = I_F \lra \Gal(H/F)$ 
is just $I^2P_F$, that is, the group of all ideals $\fa$ that 
can be written in the form $\fa = (\alpha)\fb^2$. The 
ray class group associated to $H(F)$ is 
$I/I^2P_F \simeq \Cl(F)/\Cl(F)^2$,
and the Artin symbol induces a perfect pairing 
$$  \Cl(F)/\Cl(F)^2 \times \Sel_4^+(F) \lra \mu_2, $$
where $\mu_2$ denotes the group of square roots of $1$.
This pairing is an explicit form of the decomposition law
in $H(F)/F$: a prime ideal $\fp$ in $F$ splits completely
in $H(F)$ if and only if $(H(F)/F,\fp) = 1$, i.e.,
if and only if there is some ideal $\fb$ such that 
$\fp\fb^{-2}$ is principal.

Of course we get similar results for the other Selmer groups:

\begin{theorem}\label{TPP}
Let $F$ be a number field. Then the pairings
\begin{align}
\label{EP1}   \Cl_F^+\{4\}/\Cl_F^+\{4\}^2 \times \Sel(F) & \lra \mu_2 \\
\label{EP2}   \Cl_F\{4\}/\Cl_F\{4\}^2 \times \Sel^+(F)   & \lra \mu_2 \\
\label{EP3}   \Cl^+(F)/\Cl^+(F)^2    \times  \Sel_4(F)   & \lra \mu_2 \\
\label{EP4}     \Cl(F)/\Cl(F)^2      \times \Sel_4^+(F)  & \lra \mu_2 
\end{align}
are perfect.
\end{theorem}

The claims in this theorem are equivalent to Hecke's 
theorem 171, 173, 172, and 170, respectively. For totally
complex fields with odd class number they were first proved
by Hilbert \cite[Satz 32, 33]{HRQ}, and the general proofs
are due to Furtw\"angler \cite{Fu2}.

\subsection{Selmer Groups and Quadratic Reciprocity}
Reciprocity laws may be interpreted as decomposition laws in 
abelian extensions; deriving explicit formulas from such general
results is, however, a nontrivial matter. Hecke proved an explicit
quadratic reciprocity law in number fields using quadratic Gauss 
sums, and then derived the existence of $2$-class fields from his 
results. Furtw\"angler, on the other hand, used the existence of 
class fields to derive the reciprocity law:

\begin{theorem}\label{TQR}[Quadratic Reciprocity Law]
Let $F$ be a number field, and let $\alpha, \beta \in \cO_F$
be coprime integers with odd norm. Assume moreover that 
$\alpha$ and $\beta$ have coprime conductors. Then 
$(\frac{\alpha}{\beta}) = (\frac{\beta}{\alpha})$. 
\end{theorem}

The conductor of $\alpha \in F^\times$ is by definition the conductor 
of the quadratic extension $F(\sqrt{\alpha}\,)/F$. Sufficient conditions
for coprime integers $\alpha, \beta \in \cO_F$ to have coprime conductors
are
\begin{itemize}
\item $\alpha$ is primary and totally positive;
\item $\alpha$ is primary and $\beta$ is totally positive.
\end{itemize}
By the last pairing in Theorem \ref{TPP}, we see that
$(\frac{\omega}{(\alpha)}) = 1$ for all $\alpha \in F^\times$
and $\omega \in \Sel_4^+(F)$; in particular, we have
$(\frac{\omega}{\fa}) = (\frac{\omega}{\fb})$ for all ideals
$\fa \sim \fb$ in the same ideal class, or, more generally, for all 
ideals in the same coset of $\Cl(F)/\Cl(F)^2$. By applying this
observation to certain quadratic extensions of $F$, Furtw\"angler
was able to prove the quadratic reciprocity law in $F$:

\begin{proof}[Proof of Theorem \ref{TQR}]
Put $K = F(\sqrt{\alpha\beta})$; then 
$L = K(\sqrt{\alpha}) = K(\sqrt{\beta})$, and since 
$\alpha\Fts \in \Sel_4^+(F)$ and $(\alpha,\beta) = 1$, 
we conclude that $L/K$ is unramified everywhere.
Now $\alpha \cO_L = \fa^\ell$ and $\beta\cO_L = \fb^\ell$; 
moreover $\fa \sim \fb$ since these ideals differ by the 
principal ideal generated by $\sqrt{\alpha\beta}$. Let 
$\fc$ be an ideal in $[\fa] \in \Cl(K)$ that is coprime to 
$2\fa\fb$. Let $(\frac{\,\cdot\,}{\,\cdot\,})$ 
and $(\frac{\,\cdot\,}{\,\cdot\,})_K^{\phantom{1}}$ denote the 
quadratic residue symbols in $F$ and $K$, respectively. 
Then $(\frac{\alpha}{\beta}) = (\frac{\alpha}{\fb})_K^{\phantom{1}}$
and $(\frac{\beta}{\alpha}) = (\frac{\beta}{\fa})_K^{\phantom{1}}$  by 
\cite[Prop. 4.2.]{LemRL}, 
$(\frac{\alpha}{\fb})_K^{\phantom{1}} 
      = (\frac{\alpha}{\fc})_K^{\phantom{1}}$ since $\fb \sim \fc$, 
$(\frac{\alpha}{\fc})_K^{\phantom{1}}
      = (\frac{\beta}{\fc})_K^{\phantom{1}}$ since 
$K(\sqrt{\alpha}) = K(\sqrt{\beta})$, and by going
backwards we find  $(\frac{\alpha}{\beta}) = (\frac{\beta}{\alpha})$ 
as claimed. 
\end{proof}

\subsubsection*{The First Supplementary Law of Quadratic Reciprocity}
The fact that the pairing 
$$\Cl_F^+\{4\}/\Cl_F^+\{4\}^2 \times \Sel(F) \lra \mu_2: \ \ 
 \la [\fa], \omega \ra  \too \Big(\frac{\omega}{\fa}\Big) $$
is perfect can be made explicit as follows:

\begin{theorem}\label{TFSL}
Let $\fa$ be an integral ideal with odd norm in some number field $F$. 
Then the following assertions are equivalent:
\begin{enumerate}
\item there is an integral ideal $\fb$ and some 
      $\alpha \equiv 1 \bmod 4\infty$ such that $\fa \fb^2 = (\alpha)$;
\item we have $(\frac{\omega}{\fa}) = 1$ for all $\omega \in \Sel(F)$.
\end{enumerate}
\end{theorem}

This result is Hilbert's version of the first supplementary law of
quadratic reciprocity in number fields $F$ (see Hecke \cite[Satz 171]{Hec}).
In fact, for $F = \Q$ this is the first supplementary law of quadratic
reciprocity: since $\Q$ has class number $1$, condition (1) demands
that an ideal $(a)$ is generated by some positive $a \equiv 1 \bmod 4$;
moreover, $\Sel(\Q)$ is generated by $\omega = -1$, hence 
Theorem \ref{TFSL} states that $(\frac{-1}{(a)}) = 1$ if and only if 
$a > 0$ and $a \equiv 1 \bmod 4$ (possibly after replacing the generator
$a$ of $(a)$ by $-a$).

Similarly, the fact that the pairing   
$\Cl_F\{4\}/\Cl_F\{4\}^2 \times \Sel^+(F) \lra \mu_2$ 
is perfect is equivalent to Hecke \cite[Satz 173]{Hec}).

\section{Miscellanea}
Apart from the applications of Selmer groups discussed in 
Section \ref{Sapp}, the results presented so far go back
to Hecke. In this section we will describe a few developments
that took place afterwards. Unfortunately, reviewing e.g. 
Oriat's beautiful article \cite{Ori} and the techniques 
of Leopoldt's Spiegelungssatz (reflection theorem) would 
take us too far afield.

\subsection{Reciprocity Laws}
In \cite{KS}, Knebusch \& Scharlau presented a simple proof of
Weil's reciprocity law based on the theory of quadratic forms
and studied the structure of Witt groups. In their investigations, 
they came across a group they denoted by $P/\Fts$, which 
coincides with our $\Sel(F)$, and they showed that 
$\dim \Sel(F) = \rho+r+s$ in \cite[Lemma 6.3.]{KS}. In the 
appendix of \cite{KS}, they studied $\Delta^+ = \Sel_4^+(F)$ 
and proved that the pairings (\ref{EP3}) and (\ref{EP4}) are 
nondegenerate in the second argument. 

Kolster took up these investigations in \cite{Kol} and 
generalized the perfect pairings above to certain general 
class groups and Selmer groups whose definitions depend on 
a finite set of primes $S$. In fact, let $F$ be a number field, 
and let $v_\fp(\alpha)$ denote the exponent of $\fp$ in the 
prime ideal factorization of $(\alpha)$. For finite sets $S$ 
of primes in $F$ containing the set $S_\infty$ of infinite 
primes, let $I_S$ denote the set of all ideals coprime to 
the finite ideals in $S$ and to all dyadic primes, and define
\begin{align*}
 D(S) & = \{\alpha \in F^\times: v_\fp(\alpha) \equiv 0 \bmod 2
         \ \text{for all} \ \fp \not\in S\}/\Fts,  \\
 R(S) & = I_S^2/I_S^2 \cdot \{(\alpha) \in P_4: 
            \alpha \in F_\fp^{\times\,2} \ \text{for}\ \fp \in S\}.
\end{align*}
Then $D(S_\infty) = \Sel(F)$ and $R(S_\infty) = I^2/I^2P_4^+$.
One of Kolster's tools is the perfect pairing (see \cite[p. 86]{Kol})
$$ D(S) \times R(S) \lra \mu_2, $$
which specializes to the perfect pairing (\ref{EP1}) in 
Thm. \ref{TPP} for $S = S_\infty$. 

We also remark that Kahn \cite{Kahn} studied connections between
groups related to Selmer groups and pieces of the Brauer group
Br$(\Q)$.

\subsection{Capitulation}
Let $L/K$ be an extension of number fields. Then the maps
\begin{align*}
 j_{K \to L}: & \Sel(K) \lra \Sel(L); j(\alpha \Kts) = \alpha \Lts \\
\intertext{and}
 N_{L/K}:     & \Sel(L) \lra \Sel(K); N(\alpha \Lts) = N(\alpha) \Kts
\end{align*}
are well defined homomorphism. Since $N_{L/K} \circ j_{K \to L}$
is raising to the $(L:K)$-th power, $j_{K \to L}$ is injective
and $N_{L/K}$ is surjective for all extensions $L/K$ of odd degree.

For extensions of even degree, on the other hand, these maps
have, in general, nontrivial kernels and cokernels. 
In fact, for $K = \Q(\sqrt{10}\,)$ we have 
$$ \Sel_4(K) = \Sel^+(K) = \Sel_4^+(K) = \la 5 \ra, $$
and in the quadratic extension $L = \Q(\sqrt{2},\sqrt{5}\,)$ we have 
$$ \Sel_4(L) = \Sel^+(L) = \Sel_4^+(L) = 1 $$ 
because $L$ has class number $1$ in the strict sense.
More generally, it is obvious that e.g. $\Sel_4^+(K)$ capitulates in the
extension $L = K(\sqrt{\Sel_4^+(K)}\,)$. 

Note that $E^+/E^2 \hra \Sel^+(F)$. Garbanati \cite[Thm. 2]{Gar2}
observed that, for extensions $L/K$ of totally real number fields,
$E_L^+ = E_L^2$ implies that $E_K^+ = E_K^2$. This was generalized
by Edgar, Mollin \& Peterson \cite{EMP}:

\begin{prop}
Let $L/K$ be an extension of totally real number fields. Then 
$\dim E_K^+/E_K^2 \le \dim E_L^+/E_L^2$.
\end{prop}

\begin{proof}
Let $F^1$ and $F_+^1$ denote the Hilbert class fields of $F$
in the usual and in the strict sense. Then it is easily checked that 
$K^1 = K^1_+ \cap L^1$. This implies that 
$(K^1_+ L^1:L^1) = (K^1_+ : K^1)$, hence
$(K^1_+ : K^1) \mid (L^1_+ : L^1)$. 
\end{proof}

Thus although the image of $E_K^+/E_K^2$ in $E_L^+/E_L^2$ can become
trivial (and will be trivial if and only if $L$ contains 
$K(\sqrt{E^+}\,)$), the dimension of $E_L^+/E_L^2$ cannot decrease.
Something similar does not hold for $E_4^+$: in $K = \Q(\sqrt{34}\,)$,
we have $E_4^+/E^2 = \la \eps E^2\ra$ for $\eps = 35 + 6\sqrt{34}$. 
The field $L = K(\sqrt{\eps}\,) = \Q(\sqrt{2},\sqrt{17}\,)$ 
hs class number $1$, hence $(E_L)_4^+ = E_L^2$.

\subsection{Galois Action}
Oriat \cite{Ori} (see also Taylor \cite{Tay}) derived, by applying 
and generalizing Leopoldt's Spiegelungssatz, a lot of nontrivial 
inequalities between the ranks of pieces of the class groups in 
the usual and the strict sense. As a special case of his general 
result he obtained the following theorem:

\begin{theorem}\label{TO}
Let $K/\Q$ be a finite abelian extension of number fields. Assume
that the exponent of $G = \Gal(K/\Q)$ is odd, and that 
$-1 \equiv 2^t \bmod n$ for some $t$. Then 
$$\rho^+ = \rho, \quad \dim E^+/E^2 \le \rho, \quad 
                       \dim E_4/E^2 \le \rho.$$
\end{theorem}

Observe that this contains Prop. \ref{POs} as a very special case.

In \cite{EMP} it is erroneously claimed that Oriat proved this 
theorem for general (not necessarily abelian) extensions; the
authors also give a proof of Theorem \ref{TO} ``in the abelian 
case'' which is based on the techniques of Taylor \cite{Tay}.

\end{document}